\documentclass{llncs}
\usepackage{llncsdoc}
\usepackage{amssymb, url}
\usepackage{amsmath,color}
\usepackage{epsfig}
\usepackage{enumerate}

\def\jdlqed{\vbox{\hrule \hbox{\vrule\hbox to
5pt{\vbox to 6pt{\vfil}\hfil}\vrule}\hrule}}
\newcommand{\cqfd}
{%
\mbox{}%
\nolinebreak%
\hfill%
\jdlqed
\medbreak%
\par%
}

\newcommand{\real}{\mathbb R}
\newcommand{\R}{\mathbb R}

\newcommand{\Z}{\mathbb Z}

\DeclareMathOperator{\conv}{conv}

\begin{document}

\title{A Quantitative Doignon-Bell-Scarf Theorem}

\author{Iskander Aliev, Robert Bassett, Jes\'us A. De Loera, and Quentin Louveaux}

\institute{Cardiff University, UK,\\
\email{AlievI@cardiff.ac.uk},\\
\and
University of California, Davis\\
\email{rbassett@math.ucdavis.edu}\\
\email{deloera@math.ucdavis.edu}\\
\and
Universit\'e de Li\`ege, Belgium \\
\email{q.louveaux@ulg.ac.be}}

\date{\today}

\maketitle

\begin{abstract}
The famous Doignon-Bell-Scarf theorem is a Helly-type result about the existence of integer solutions to systems of linear inequalities.
The purpose of this paper is to present the following quantitative generalization:   Given an integer $k$, we prove that there exists a
constant $c(n,k)$, depending only on the dimension $n$ and $k$,
such that if a bounded polyhedron $\{x \in \R^n : Ax \leq b\}$ contains exactly $k$ integer points, then there exists a
subset of the rows, of cardinality no more than $c(n,k)$, defining a polyhedron that contains exactly the same $k$ integer points.
In this case $c(n,0)=2^n$ as in the original case of Doignon-Bell-Scarf for infeasible systems of inequalities.
We work on both upper and lower bounds for the constant $c(n,k)$ and discuss some consequences, including a Clarkson-style algorithm
to find the $l$-th best solution of an integer program with respect to
the ordering induced by the objective function.
\end{abstract}

\section{Introduction}

In a \emph{Helly-type theorem},  there is a family of objects $F$, a property $P$ and a constant $\mu$ such
that if every subfamily of $F$ with $\mu$ elements has property $P$, then the entire family has property $P$. This topic expands a large literature, we recommend \cite{baranykatchalskipach,baranykatchalskipach2,danzergrunbaumklee,eckhoff,wenger} and the references there for a glimpse of this fertile subject.
The classical theorem of  Helly deals with the case when $F$ is a finite family of convex sets in $\R^n$, the constant $\mu$ is  $n+1$,
and the property $P$ is  having a non-empty intersection.  One of the most famous Helly-type theorems, due to its many applications in the theory of integer programming  and computational geometry of numbers, is the  1973 theorem of Doignon \cite{doignon},  later reproved by Bell and Scarf \cite{bell,scarf}.
For an arbitrary matrix $A\in \mathbb R^{m\times n}$ and a set $S\subseteq \{1,\ldots, m\}$,
we denote $A_S\in \mathbb R^{|S|\times n}$ as the submatrix  of $A$ with row-indices in $S$.
\vskip .3cm

\noindent {\bf Theorem} {\em Let $A$ be a $m \times n$ real matrix and $b$ a vector in $\R^m.$
 If the set of integer points $\{x \in \Z^n : Ax \leq b\}$ is empty, then there is a subset $S$ of the rows of $A$, of cardinality
 no more than $2^n$, with the property that the set $\{x \in \Z^n :A_{S} x \leq b_{S}\}$ is also empty.
}
\vskip .3cm

It should be noted that still during the 1970's, in \cite{hoffman} Hoffman created an abstract framework that contains Helly's original theorem as well as the Doignon-Bell-Scarf results.  Here our key contribution is to prove a weighted or quantitive generalization of Doignon-Bell-Scarf's theorem, one that is close in spirit to the quantitative versions of Helly's theorem of B\'ar\'any, Katchalski and Pach \cite{baranykatchalskipach,baranykatchalskipach2}:

 \begin{theorem} \label{k-doignon}
Given $n$, $k$ two non-negative integers there exists a universal constant $c(n,k)$, depending only on $n$ and $k$, such that
for any polyhedron $P=\{x \in \R^n : Ax \leq b\}$ with  exactly $k$ integer points,
there is a subset $S$ of the rows of $A$, of cardinality  no more than $c(n,k)$, with the property that the polyhedron
$\{x \in \R^n: A_S x \leq b_S \}$ has exactly the same $k$ integer points  as $P$.

We refer to $ c(n,k)$ as the minimal such constant.
The upper bound we prove in this paper is $c(n,k) \leq \lceil 2(k+1)/3\rceil 2^n-2\lceil 2(k+1)/3\rceil+2.$
\end{theorem}



Our proof of this theorem (presented in the first section of this paper) generalizes the proof of the Doignon-Bell-Scarf theorem
given in \cite{bell,schrijver} with some counting twists. Our new bound in Theorem \ref{k-doignon}  is a considerable improvement
on the  bound we presented in \cite{ipcoversion} (where the existence of the universal constant $c(n,k)$ was first announced).
The new proof is presented in Section \ref{sectionDoignon}.

\vskip .3cm
We were also able to obtain lower bounds for the case $k=1$ proving that our upper bound is in fact the exact value of $c(n,1)$. This is
the content of the following theorem. An explicit construction of the polytope in question will be presented in Section \ref{lowerb} of our paper:
\vskip .3cm

\begin{theorem}
There exists a polytope $P$ in $\R^n$ that has exactly one interior integer point, $2(2^n-1)$ facets and
one integer point in the relative interior of each facet. Thus all inequalities in $P$ are necessary in the sense that
the removal of any inequality from $P$ results in the inclusion of at least one additional integer point in the interior of $P$.
As a consequence, the upper bound of Theorem \ref{k-doignon} is tight for $k=1$ and thus $c(n,1)=2(2^n-1)$.
\label{lowerboundtheorem}
\end{theorem}

This implies that the upper bound given in Theorem \ref{k-doignon} for $k=1$ is tight, i.e. $c(n,1)=2(2^n-1).$
Also observe that the upper bounds for $c(n,2)$ and $c(n,1)$ are equal
in Theorem \ref{k-doignon}. It is an interesting problem to find exact values of $c(n,k)$ for $k\geq 2$.
Given the available data, a natural conjecture is that our
bound for $c(n,2)$ is also tight, but we also know that for $c(2,k)$ the bound is not tight for $k\geq 3$.

Finally, in Section \ref{apps} of the paper we discuss some consequences of  Theorem \ref{k-doignon}.
Using standard procedures we extend the version for half-spaces to families of convex sets and present some quantitative corollaries of Theorem \ref{k-doignon}.
K. Clarkson was the first to notice that using Doignon-Bell-Scarf's theorem (or any similar Helly-type intersection result) one can obtain a probabilistic algorithm for integer linear programming \cite{clarkson}. Clarkson's probabilistic algorithm works in great generality for so called \emph{violator spaces} \cite{gaertneretal}.
The remarkable fact is that the running time for the algorithm is {\em linear} in the number of constraints defining the problem, and {\em subexponential} in the dimension  of the problem. In this way, when the dimension of the problem is small, then these are the most efficient methods to solve a large family of optimization problems.
Using violator spaces we present an application of our quantitive Doignon-Bell-Scarf theorem to the problem of finding the $l$ best integer feasible points of an integer program.

We conclude with the remark that our results are part of the fertile and classical study of convex polyhedra with fixed number of (interior) lattice points; a topic that plays an important role in algebra, discrete geometry, and optimization. Indeed, there has been a lot of work, going back to classical results of Minkowski and van der Corput,
 to show that the volume of a lattice polytope $P$ with $k={\rm card}({\rm int}\, P \cap {\Z}^n )\geq1$ is bounded above by a constant
that only depends on $n$ and $k$ (see e.g., \cite{lagariasziegler,pikhurko}). Similarly, the supremum of the possible number of points of ${\Z}^n$
in a  lattice polytope in $\real^n$ containing precisely $k$ points of $\Z^n$ in its interior, can be bounded by a constant that only depends
on $n$ and $k$. Such results play an important role in the theory of toric varieties and the structure of lattice polyhedra (see e.g., \cite{barvinokpommersheim} and the
references therein).  In optimization it was recently shown in \cite{lpf,andersenetal} how lattice-free polyhedra can be applied to the
generation of cutting planes (see also \cite{lpfsurvey}). One can think of lattice-free polyhedra as dual certificates of the infeasibility of a
system of the form $\{x \in \Z^n : Ax=b, Cx \leq d\}$ \cite{ipfarkaslemma}. These convex bodies play an important role in the modern theory of cutting planes \cite{borozancornuejols,deywolsey}.  There has been some interesting hybrid variations of Helly and Doignon-Bell-Scarf in \cite{averkovweismantel,hoffman}.

\section{A Generalization of Doignon-Bell-Scarf's theorem}
\label{sectionDoignon}

In this section we will prove Theorem \ref{k-doignon}. The constant $c(n,k)$ we provide is $\lceil 2(k+1)/3\rceil 2^n-2\lceil 2(k+1)/3\rceil+2$. This is an
improvement from the constant presented in an earlier conference version of this paper \cite{ipcoversion}.

\vskip .3cm



\noindent [{\em Proof of Theorem} \ref{k-doignon}]

The proof proceeds by contradiction.  We choose a counterexample system with minimal number $m$ of  linear inequalities,
\begin{equation}
a_1x\le \beta_1,\ldots, a_m x\le \beta_m, x \in \R^n\,.
\label{system}
\end{equation}

This counterexample system  (\ref{system}) defines a polyhedron with exactly $k$ integer solutions 
and $m\geq \lceil 2(k+1)/3\rceil 2^n-2\lceil 2(k+1)/3\rceil+3$. Again, remember
that $m$ is smallest possible among all counterexamples to Theorem  \ref{k-doignon}. The minimality of the chosen counterexample implies that
if we delete any of the constraints in (\ref{system}), the remaining system has at least $k+1$ integer solutions (one is not able to remove any of the constraints
otherwise that would produce a smaller counterexample). This says that there exist integer vectors $x_1, \ldots, x_m$ such that $x_j$ violates $a_jx \le \beta_j$, but
satisfies all other inequalities in (\ref{system}). Let $Q$ be a cube in $\R^n$ that contains all integer solutions to (\ref{system}) and all points  $x_1, \ldots, x_m$.
%
Consider the set of integer points
\begin{equation*}
G=Q\cap\Z^n\,.
\end{equation*}

Since the polyhedron $Q$ is bounded, the set $G$ is finite.
Now we will associate a new system to (\ref{system}) in a special critical shape constructed  by shifting its hyperplanes.
Consider the set $\Gamma\subset \R^m$ of the vectors $(\gamma_1, \ldots, \gamma_m)$ such that
\begin{equation}\label{facet}
\gamma_j\ge \min\{a_j z : z\in G, a_j z >\beta_j\}
\end{equation}
and
\begin{equation}\label{notmany}
\text{the system }
a_1x<\gamma_1,\ldots, a_mx<\gamma_m
\text{ has exactly }k\\ \text{ integer solutions in }G.
\end{equation}



The set $\Gamma$ is nonempty as we can take the equality in (\ref{facet}).  Next notice that  (\ref{notmany}), together with the
lower bounds on the $\gamma_i$, implies that any integer solution of the system (\ref{system}) remains feasible for the system
$a_1x<\gamma_1,\ldots, a_mx<\gamma_m$ for $\gamma \in \Gamma$. Thus, for all $\gamma \in \Gamma$, $a_1x<\gamma_1,\ldots, a_mx<\gamma_m$
share exactly the same $k$ integer solutions in $G$ as (\ref{system}).

Observe also that the set $\Gamma$ is bounded. Otherwise, $\gamma_j$ for some $j$ grows arbitrarily large. In particular $\gamma_j$ is much larger than $a_j x_j$, for  the
point $x_j\in G$ we associated with the constraint $a_jx \leq \beta_j$ in (\ref{system}). Note $x_j$ is by  construction not a solution of system (3). But  by the size of $\gamma_j$,
the point  $x_j$ satisfies all the inequalities $a_1x<\gamma_1,\ldots, a_m x<\gamma_m$.  Hence, $x_j$ is an additional integer feasible point for (\ref{notmany}), which is a contradiction.

\begin{claim} There is a point $(\nu_1, \ldots, \nu_m)\in \Gamma$ such that
\begin{equation}
\mbox{for each }j=1,\ldots,m \mbox{ there exists }y_j\in G \mbox{ so that }a_j y_j =\nu_j
\mbox{ and }a_i y_j <\nu_i \, (i\neq j)\,.
\label{extreme_point}
\end{equation}
\end{claim}

Before a formal proof of the Claim there is a an intuitive justification for it. Since the set $G$ is finite, we can define the numbers $\nu_1, \ldots, \nu_m$ satisfying the Claim by changing the right-hand side of the inequalities one by one (i.e., shifting the hyperplanes), while controlling each time how many points of $G$ are within the new polyhedron.

\textit{Proof of Claim}: Take any point $(\nu_1, \ldots, \nu_m)\in \Gamma$ and suppose that for some $j$ this property does not hold. Consider
\begin{equation}
\nu'_j=\sup\{\nu: (\nu_1,\ldots,\nu_{j-1},\nu, \nu_{j+1}, \ldots, \nu_m)\in \Gamma\}\,.
\label{max_gamma}
\end{equation}
The supremum in (\ref{max_gamma}) is finite as the set $\Gamma$ is bounded. If $(\nu_1,\ldots,\nu_{j-1},\nu'_j, \nu_{j+1}, \ldots, \nu_m)\notin \Gamma$ then, by (\ref{notmany}),
for any $\delta>0$ there should exist a point $z\in G$ such that $\nu'_j-\delta\le a_jz<\nu'_j$.
This is impossible as $G$ is finite. Consequently, $(\nu_1,\ldots,\nu_{j-1},\nu'_j, \nu_{j+1}, \ldots, \nu_m)\in \Gamma$.
Observe that there should exist $y_j\in G$ with $a_j y_j =\nu'_j$ and $a_i y_j <\nu_i (i\neq j)$. Otherwise  $(\nu_1,\ldots,\nu_{j-1},\nu'_j+\epsilon, \nu_{j+1}, \ldots, \nu_m)\in \Gamma$ for some $\epsilon>0$ as $G$ is a finite set. Hence we can replace $\nu_j$ by $\nu'_j$. After at most $m$ such replacements
we will construct a point satisfying (\ref{extreme_point}). \cqfd

By the established Claim, the convex hull $H=\conv(\{y_1, \ldots, y_m\})$ contains at most $k$ integer points distinct from $y_1, \ldots, y_m$.

%



The property of the set $\{y_1, \ldots, y_m\}$ expressed by (\ref{extreme_point}) is very important and as we will use it several times later, we formally name it.

\begin{definition}
Let $X$ be a finite set in $\Z^n$. We say that $X$ satisfies the \emph{support hyperplane property}
if for every $y\in X$, there exists a half-space $f^Tx\leq g$ such that $f^Ty=g$ and
$f^Tz<g$ for every $z\in X, z\neq y.$
Furthermore, we say that the inequality $f^Tx\leq g$ fulfills the support hyperplane property for $y$.
\end{definition}
Observe that the support hyperplane property is equivalent to saying that all members of $X$ are
vertices of $\conv(X)$.

The application of the following four lemmas directly gives the desired contradiction and finishes the proof of the Theorem:

\begin{lemma}
Consider a finite set $X\subset \Z^n$ with $|X|\geq 2^n+1$
that satisfies the support hyperplane property.
Then there exists a  point $z\in (\conv(X)\cap \Z^n)\setminus X$.
\label{claim22}
\end{lemma}
\textit{Proof of Lemma}:
Since $|X|\geq 2^n+1$, by the pigeonhole principle
there exist $y_{1},y_{2}\in X$ with $y_{1}\neq y_{2}$ and $y_{1} \equiv y_{2}(\text{mod}\ 2).$
Therefore $z=\frac{1}{2}(y_{1}+y_{2})\in \Z^n$. Since $X$ satisfies the support hyperplane property, we conclude that $z\notin X$. 

\begin{lemma}
Consider a finite set $X\subset \Z^n$ that satisfies the support hyperplane property.
Suppose there exists a point $z\in (\conv(X)\cap \Z^n)\setminus X$.
Then there exist disjoint subsets $X_1, X_2\subseteq X$
with $X_1 \cup X_2=X$ such that both
$X_1\cup \{z\}$ and $X_2\cup \{z\}$ satisfy the support hyperplane property.
\label{claim23}
\end{lemma}

\textit{Proof of Lemma:}  There exists a hyperplane defined by the equation $\bar f^T x=\bar g$ such that $\bar f^T z= \bar g$
and the equality does not hold for any other member of $X$.
We can split the other members of $X$ into two disjoint sets $X_1=X\cap \{x\in \R^n :\bar f^Tx<\bar g\}$
and $X_2=X\cap \{x\in \R^n :\bar f^Tx>\bar g\}$.

The result follows since for every $x\in X_i$, the inequality fulfilling the
support hyperplane property for $x$ in $X$ fulfills  the support hyperplane property for $x$ in $X_i \cup \{z\}$.
The inequality $\bar f^Tx\leq \bar g$ (respectively $\bar f^Tx\geq \bar g$) fulfills the support hyperplane property
for $z$ in $X_1\cup \{z\}$ (respectively in $X_2\cup \{z\}$). \cqfd

\begin{lemma}
Consider a finite set $X\subset \Z^n$ that satisfies the support hyperplane property.
Suppose  $z_1, z_2\in \conv(X)\cap \Z^n$ with $z_1,z_2\notin X$ and $z_1\neq z_2$.
There exist disjoint subsets $X_1, X_2\subseteq X$
with $|X_1| + |X_2|\ge |X|-2 $ such that both
$X_1\cup \{z_1, z_2\}$ and $X_2\cup \{z_1,z_2\}$ satisfy the support hyperplane property.
\label{twopoints}
\end{lemma}

\textit{Proof of Lemma:}
Since $X$ satisfies the support hyperplane property, the line $L$
spanned by $z_1, z_2$ intersects $X$ in at most two points.
Consequently, there exists a hyperplane $\bar f^T x=\bar g$ such that $\bar f^T z_1= \bar f^T z_2= \bar g$
and the equality holds for at most two points of $X$ (precisely for the points in $X\cap L$).
We can split the points of the set $Y=X\setminus (X\cap L)$ into two disjoint sets $X_1=Y\cap \{x\in \R^n : \bar f^Tx<\bar g\}$
and $X_2=Y\cap \{x\in \R^n : \bar f^Tx>\bar g\}$. Clearly, $|X_1| + |X_2|\ge |X|-2$.

Let us now show that $X_1\cup \{z_1, z_2\}$ and $X_2\cup \{z_1,z_2\}$ satisfy the support hyperplane property.
First, for every $x\in X_i$, the inequality fulfilling the
support hyperplane property for $x$ in $X$ fulfills  the support hyperplane property for $x$ in  $X_i \cup \{z_1, z_2\}$.
Next, both points $z_1, z_2$ are on the hyperplane $\bar f^T x=\bar g$ and this hyperplane does not intersect either of the finite sets $X_1, X_2$.
Thus applying a sufficiently small rotation of the  hyperplane $\bar f^T x=\bar g$ around $z_1$ we can obtain a hyperplane $\bar f_1^T x=\bar g_1$
such that  $\bar f_1^T z_1=\bar g_1$ and $\bar f_1^T y<\bar g_1$ for any point  $y\in X_1\cup\{z_2\}$, so
that the constructed hyperplane fulfills the support hyperplane property for $z_1$ in $X_1\cup \{z_1, z_2\}$. In the same manner we can construct
the hyperplanes fulfilling the support hyperplane property for $z_1$ in $X_2\cup \{z_1, z_2\}$ and for $z_2$ in $X_i\cup \{z_1, z_2\}$, $i=1,2$.\cqfd

\begin{lemma}
Let $n\geq 2$ and $k$ be  natural numbers.
Consider a finite set $X\subset \Z^n$ and assume that it satisfies the support
hyperplane property.
If the cardinality $|X|\geq k2^n-2k+3,$ then there exist at least $\lfloor 3k/2\rfloor$ different integer
points in $\conv(X)\setminus X$.
\label{main_lemma}
\end{lemma}

\textit{Proof of Lemma:}
We will proceed by induction on $k$. For $k=1$, the result follows from Lemma \ref{claim22}.

We now assume that the result is true up to $k-1$ and prove it for $k\geq 2$.
Assume that $|X|\geq k2^n-2k+3$. We have $|X|\geq 2^n+1$ and therefore by
Lemma \ref{claim22}, there exists
a point $z_1\in \mathbb Z^n$ in $\conv(X)\setminus X$.

By Lemma \ref{claim23}, applied to $X$ and $z_1$, there exist disjoint subsets $Y_1, Y_2\subseteq X$
with $Y_1 \cup Y_2=X$ such that both
$Y_1\cup \{z_1\}$ and $Y_2\cup \{z_1\}$ satisfy the support hyperplane property.

Assume without loss of generality that $|Y_2|\ge |Y_1|$. Since the sets $Y_1, Y_2$ are disjoint, $Y_2$ has cardinality at least $\lceil \frac{|X|}{2}\rceil\ge 2^n$ (note that $k\ge 2$).
By Lemma \ref{claim22}, applied to $Y_2\cup\{z_1\}$, there exists a point $z_2\in \mathbb Z^n$ in $\conv(Y_2\cup\{z_1\})\setminus (Y_2\cup\{z_1\})$. In particular, $z_1\neq z_2$.
Replacing, if necessary, $z_1, z_2$ by adjacent integer points in $\conv(\{z_1, z_2\})$, we may also assume that ${\rm card}(\conv(\{z_1, z_2\})\cap \Z^n)=2$.

Next, by Lemma \ref{twopoints}, applied to $X$ and the constructed points $z_1, z_2$, there exist disjoint subsets $X_1, X_2\subseteq X$
with $|X_1| + |X_2|\ge |X|-2 $ such that both
$X_1\cup \{z_1, z_2\}$ and $X_2\cup \{z_1,z_2\}$ satisfy the support hyperplane property.
It follows from the proof of Lemma \ref{twopoints} that if one of the sets, say $X_1$, is empty, then $|X_2\cup \{z_1,z_2\}|=|X|$.
In this case it is enough to prove the lemma for the set $X$ replaced by the set $X_2\cup \{z_1,z_2\}$. Since the set $\conv(X)\cap \Z^n$ is finite, after a finite
number of replacements we obtain the nonempty sets $X_1$, $X_2$.

Assume now without loss of generality $0<|X_1|\le |X_2|$. We will consider the following two cases. First,
suppose that there exists $l\in\mathbb Z$ with $1\leq l\leq k$ such that
\begin{equation}
(l-1)2^n-2(l-1)+1 \leq |X_1| \leq l 2^n-2l.
\label{x1 cardinality}
\end{equation}
Since $X_1$ and $X_2$ are disjoint subsets of $X$ with $|X_1| + |X_2|\ge |X|-2$ and since $|X|\geq k2^n-2k+3,$
it follows from the upper bound in  \eqref{x1 cardinality} that
\begin{align}
|X_2| & \geq k2^n-2k+1 - l2^n+2l \notag\\
&= (k-l)2^n-2(k-l)+1. \label{x2 cardinality}
\end{align}
Thus using \eqref{x1 cardinality},
we have $|X_1\cup\{z_1,z_2\}|\geq (l-1)2^n-2(l-1)+3$ which implies
from the induction hypothesis (note that $l\leq k$) that there are $\lfloor 3(l-1)/2\rfloor$ additional integer points
in $\conv(X_1\cup\{z_1,z_2\})\setminus(X_1\cup\{z_1,z_2\})$.
Using \eqref{x2 cardinality}, we have $|X_2\cup\{z_1,z_2\}|\geq (k-l)2^n-2(k-l)+3$
which, from the induction hypothesis (note that $l\geq 1$) implies that there are
$\lfloor 3(k-l)/2\rfloor$ additional integer points in $\conv(X_2\cup\{z_1,z_2\})\setminus(X_2\cup\{z_1,z_2\})$.
The result follows since, counting $z_1, z_2$, we have provided $\lfloor 3(l-1)/2\rfloor+\lfloor 3(k-l)/2\rfloor+2\ge \lfloor 3k/2 \rfloor$ different integer points
in  $\conv(X)\setminus X$. Suppose now that $|X_1|\ge k 2^n-2k+1$. Then the set $X_1\cup\{z_1,z_2\}$  has cardinality at least $k 2^n-2k+3$. Hence, similarly to the previous case, the result follows from
the induction hypothesis.\cqfd

It should be emphasized here that Lemma \ref{main_lemma} is
of independent interest for the theory of lattice points in convex lattice polytopes.  It can be restated as follows.
\begin{corollary}
Let $n\geq 2$ and $k$ be  natural numbers.
Consider a convex lattice polytope  $P$ with the set of vertices $X\subset \Z^n$.
If the cardinality $|X|\geq k2^n-2k+3,$ then there exist at least $\lfloor 3k/2\rfloor$ different lattice
points in $P\setminus X$.
\end{corollary}

If we go back to the proof of Theorem \ref{k-doignon}, recall that existence of the counterexample system (\ref{system}) implies the existence of
 $m$ integer points $y_1,\ldots, y_m$ satisfying the support hyperplane property, with at most $k$ other integer points in their convex hull $H$
and with
$m\geq \lceil 2(k+1)/3\rceil 2^n-2\lceil 2(k+1)/3\rceil+3$.


Applying Lemma \ref{main_lemma}, we conclude that there must be at least
\begin{equation}
\left\lfloor \frac{3}{2}\left\lceil \frac{2(k+1)}{3} \right\rceil \right\rfloor
\label{complicatedfloor}
\end{equation}
other integer points in $H$. Observe that \eqref{complicatedfloor} is equal to $k+1$ if $k$ is congruent to 0 or 2 modulo 3
and equal to $k+2$ if $k$ is congruent to 1 modulo 3. Hence there are at least $k+1$ other integer points in $H$
which is the desired contradiction.\cqfd

\newcommand{\abs}[1]{\left| #1 \right|}
\renewcommand{\vec}[1]{\mathbf{#1}}

\section{ Lower Bound Constructions for $k=1$}
\label{lowerb}

In page 235 of \cite{schrijver}, an example is given that shows that the bound $2^n$ given by the Doignon-Bell-Scarf
theorem is tight. In this section, we present a construction of a polytope $P$ that shows that
our upper bound for $k=1$ from Theorem \ref{k-doignon} is tight. This example, together with the verification
of its properties, establish Theorem \ref{lowerboundtheorem}.

For notational convenience, given a set of natural numbers $N$,
let $l_N:= \min_{i\in N}\; i$ denote its least element.
We define the polyhedron as follows

\begin{alignat}{2}
P=\{ x\in \mathbb R^n:\qquad \sum_{i=1}^{j-1} \frac{1}{2^{i}} x_i + x_j + \sum_{i=j+1}^{n}
\frac{1}{2^{i-1}} x_i &\leq 1 \qquad&& j=1,...,n, \label{eq1}\\
- \sum_{i=1}^{j-1} \frac{1}{2^{i}} x_i - x_j - \sum_{i=j+1}^{n}
  \frac{1}{2^{i-1}} x_i &\leq 1 \;&& j=1,...,n, \label{eq2}\\
-\frac{1}{\abs{N}} x_{l_N}+\sum_{i \in N, i \neq l_{N}} \frac{1}{\abs{N}} x_i - \sum_{i \not \in N}
\frac{1}{\abs{N}^n} x_i &\leq 1 \; \;&& \forall N \subseteq \{1,2,..,n\}; \; \; \abs{N} \geq 2, \label{eq3}\\
+\frac{1}{\abs{N}} x_{l_N}-\sum_{i \in N, i \neq l_{N}} \frac{1}{\abs{N}} x_i + \sum_{i \not \in N}
\frac{1}{\abs{N}^n} x_i &\leq 1 \; \;&& \forall N \subseteq \{1,2,..,n\}; \; \; \abs{N} \geq 2\quad \}.\label{eq4}
\end{alignat}

%

The rationale behind the construction of the polyhedron $P$ is the following.
First it is constructed in such a way that 0 is the only integer point in its interior.
Both of the inequalities \eqref{eq1}-\eqref{eq2} are tight at a unit vector $\pm e_i$
and exclude some integer points from $\{-1,0,1\}^n$.
All inequalities \eqref{eq3}-\eqref{eq4} are all tight at exactly one of the remaining integer points of $\{-1,0,1\}^n.$

We will prove Theorem \ref{lowerboundtheorem} through Lemmas \ref{lemmageq2} to \ref{allnecessary}.
Lemma \ref{lemmageq2} proves that the only valid integer points of $P$ are in $\{-1,0,1\}^n$. Lemma \ref{conditionfeasibility}
uses Lemma \ref{lemmageq2} to provide a necessary and sufficient condition of feasibility of integer points.
Lastly, Lemma \ref{allnecessary} shows that each inequality defining $P$ is necessary and contains exactly one tight integer point in its relative interior.
Because we have used only rational data the polyhedron is in fact bounded, thus a polytope. This is the case because if unbounded, its recession cone
would contain a rational direction which would force infinitely many points inside.

\begin{lemma}
If $y \in \Z^n$ has at least an index $j$ such that $\abs{y_j} \geq 2$, then $y \not \in P$
\label{lemmageq2}
\end{lemma}

\begin{proof}
Consider $y\in P\cap \mathbb Z^n.$
Assume by contradiction that $k$ is the largest index with $|y_k|\geq 2.$
We prove the case $y_k\geq 2$. The negative case is symmetric and
omitted.

Add twice Inequality \eqref{eq1} with $j=k$ to Inequality \eqref{eq2}
with $j=1$ which yields
\begin{align}
\left( 2 - \frac{1}{2^{k-1}} \right) y_k + \sum_{i=k+1}^{n}
\frac{1}{2^{i-1}} y_i \leq 3.
\label{sumoftwoineq}
\end{align}
Since, we have assumed that $|y_i|\leq 1$ for all indices $i\geq k+1$,
we can bound
\begin{align}
-\frac{1}{2^{k-1}} < \sum_{i=k+1}^n \frac{1}{2^{i-1}} y_i < \frac{1}{2^{k-1}}.
\label{boundrest}
\end{align}
Using \eqref{sumoftwoineq} and \eqref{boundrest}, we conclude that
\begin{align}
y_k < \frac{3 \cdot 2^{k-1}+1}{2^k-1}.
\label{contradic}
\end{align}
For $k\geq 3$, this provides a contradiction
since the right-hand-side of \eqref{contradic} can be shown to be smaller than 2.
For $k=1$ or $k=2$, this yields $|y_k|\leq 2.$
Observe though that if $k=1$, then the inequalities \eqref{eq1}
and \eqref{eq2} with $j=1$ together with the fact that, using $|y_i|\leq 1$ for
$i\geq 2$, $|\sum_{i=2}^n \frac{1}{2^{i-1}}y_i|< 1$ yield
$|y_1|\leq 1.$

To finish the proof, there still remains to consider the case $k=2$ i.e. $y_2=2$.
If $y_1$ is nonnegative, $y$ violates
\eqref{eq1} with $j=2$.
If $y_1$ is negative, $y$ violates \eqref{eq3}
with $N=\{1,2\}$.
%
%
\cqfd
\end{proof}

\begin{definition}
Given a point $y \in \Z^n\setminus\{0\}$, let $l(y)$ be the least nonzero index of $y$,
i.e. $y_i=0$ for all $i<l(y).$
\end{definition}

\begin{lemma}
Let $y \in \{-1,0,1\}^n$. Then $y$ is in $P$ if and only if one of the
following is true
\begin{enumerate}[(i)]
\item $y$ is the origin
\item $y_{l(y)} =1$ and $y_i \in \{-1, 0 \}$, for all  $i \geq l(y)+1$
\item $y_{l(y)} = -1$ and $y_i \in \{1, 0 \}$, for all  $i \geq l(y)+1$.
\end{enumerate}
\label{conditionfeasibility}
\end{lemma}

\begin{proof}
We first prove that if $y\in\{-1,0,1\}^n$ is feasible then it must satisfy one of the
three conditions.
Assume therefore that $y\in\{-1,0,1\}^n$ is feasible.
The point $y=0$ is trivially feasible (option (i)).
If $y$ is not the origin, there must be
some $y_j \neq 0$. Assume that $y_{l(y)} =1$.  If there is a $k\neq l(y)$ with $y_k =1$,  then $y$
violates Inequality \eqref{eq1} with $j=k$
$$\sum_{i=1}^{k-1} \frac{1}{2^{i}} x_i + x_k + \sum_{i=k+1}^{n}
\frac{1}{2^{i-1}} x_i \leq 1$$
so $y$ satisfies (ii). The case that $y_{l(y)} = -1$ is
symmetric and omitted and leads to option (iii).

Assume $y$ satisfies one of the three conditions, we want to prove that it is feasible.
Obviously, if $y$ is the origin, it is feasible.
Assume that $y$ satisfies (ii), the other case is symmetric and omitted here.
We now prove that all inequalities are satisfied by $y$.

First consider \eqref{eq1}. The term with $x_{l(y)}$ is less or equal to 1
whereas the remaining part of the summation is nonpositive which makes the
left-hand-side of \eqref{eq1} smaller or equal to 1.

Consider now \eqref{eq2}. If $j\leq l(y)$, the term with $x_{l(y)}$ is nonpositive
whereas the sum of the remaining terms is less or equal to 1 which proves that the
inequality is satisfied. If $j\geq l(y)+1$, let us denote the inequality
as $\sum_{i=1}^n \alpha_i x_i \leq 1.$ Observe that $\alpha_{l(y)}=-\frac{1}{2^{l(y)}}$,
$\alpha_j=-1$ and $0>\alpha_i\geq -\frac{1}{2^{l(y)+1}}$ for all $i\geq l(y)+1, i\neq j$.
Therefore $\alpha_{l(y)} y_{l(y)} +\sum_{i\geq l(y), i\neq j} \alpha_i x_i\leq 0$
and $\alpha_jx_j\leq 1$ which makes the
left-hand-side of \eqref{eq2} smaller or equal to 1.

Consider Inequality \eqref{eq3}.
First observe that the second sum of \eqref{eq3} is bounded from above by $1/|N|.$
Concerning the first two terms, we distinguish two cases.
If $l(y)\in N, l_N\neq l(y)$, it implies that $y_{l_N}=0$ using condition (ii),
and bounds $-\frac{1}{\abs{N}} x_{l_N}+\sum_{i \in N, i \neq l_{N}} \frac{1}{\abs{N}} x_i \leq \frac{1}{|N|}$.
Otherwise $-\frac{1}{\abs{N}} x_{l_N}\leq \frac{1}{|N|}$ and $\sum_{i \in N, i \neq l_{N}} \frac{1}{\abs{N}} x_i \leq 0$
which implies that in both cases, the left-hand-side of \eqref{eq3} is bounded from above by 1.

Consider Inequality \eqref{eq4}.
We distinguish two cases. In the first case, we assume that $l(y)=l_N$.
This implies that first term of \eqref{eq4} equals $1/|N|$, the first sum is bounded
from above by $(|N|-1)/|N|$ as it contains $|N|-1$ terms and the last sum is nonpositive
using condition (ii). Therefore the left-hand-side of \eqref{eq4} is bounded from above by 1.
In the second case, we assume that $l(y)\neq l_N$. Therefore the first term
of \eqref{eq4} is bounded from above by 0, the first sum is bounded from above by
$(|N|-1)/|N|$ and the second sum is bounded from above by $1/|N|^n$ and the result
follows.
\cqfd
\end{proof}

\begin{lemma}
Each of the $2(2^n-1)$ inequalities in $P$ is necessary, i.e., the removal of any inequality from $P$ results in the inclusion of at least one additional integer point in the interior of $P$. 
\label{allnecessary}
\end{lemma}

\begin{proof}
We will show the lemma by proving that each facet of
$P$ contains exactly one integer feasible point in its relative interior.

Consider an inequality of type \eqref{eq1}. Observe that for any point satisfying
conditions (ii) or (iii) of Lemma \ref{conditionfeasibility},
$\sum_{i=1}^{j-1} \frac{1}{2^i} x_i + \sum_{i=j+1}^{n}
\frac{1}{2^{i-1}} x_i <1$. Therefore to make \eqref{eq1} tight
we need $x_j=1$ which implies  $\sum_{i=1}^{j-1} \frac{1}{2^i} x_i + \sum_{i=j+1}^{n}
\frac{1}{2^{i-1}} x_i =0$. Since the coefficients are in a geometric
progression, this in turn implies $x_i=0$ for all $i\neq j$.
We have therefore proven that the unit vectors are tight for all inequalities
of type \eqref{eq1}. By symmetry, for each inequality of type \eqref{eq2},
only $-e_j$ (where $e_j$ denotes the $j^{th}$ unit vector) is tight,
integer  and valid.

Consider an inequality of type \eqref{eq3}. Observe that for any point
satisfying conditions  (ii) or (iii) of Lemma \ref{conditionfeasibility},
 $- \sum_{i \not \in N}
\frac{1}{\abs{N}^n} x_i < \frac{1}{|N|}.$
Therefore to make \eqref{eq3} tight, we need
$-\frac{1}{\abs{N}} x_{l_N}+\sum_{i \in N, i \neq l_{N}} \frac{1}{\abs{N}} x_i = 1$
which implies $x_{l_N}=-1$ and $x_i=1$ for all $i\in N\setminus l_N$
and $x_i=0$ for $i\not\in N.$
Symmetrically , for each inequality of type \eqref{eq4}, only
$x_{l_N}=1$, $x_i=-1$ for all $i\in N\setminus l_N$
and $x_i=0$ for $i\not\in N$  is tight, integer  and valid.

By observing that all points that were shown to be tight
for the facet-defining inequalities of $P$ are all different, the result follows.

%
Lastly, the fact that there are $2(2^n-1)$ planes follows from the fact
that they are in bijection with double the number of nonempty subsets
of $\{1,\ldots, n\}$.\cqfd
\end{proof}

In this section, we have dealt with the case $k=1$ and proven that the upper bound given in Theorem
\ref{k-doignon} is tight. Since the upper bound for $k=2$ matches that for $k=1$, it is natural to conjecture
that the bound is tight for $k=2$ as well. We know that for
$c(2,k)$ the bound is not tight for $k\geq 3$. We also believe that it is not tight for $n, k\ge 3$ and can further be improved.

\section{Consequences and variations of Theorem \ref{k-doignon}}  \label{apps}

In this section, we discuss variations and consequences of Theorem \ref{k-doignon}.

Let us begin by remarking that while replacing the $=$ by $\geq k$ in the statement of Theorem \ref{k-doignon} gives a very easy-to-prove result, nevertheless one can state a more surprising corollary of Theorem \ref{k-doignon} that involves estimations of the number of integer points and  resembles more the quantitative Helly theorem of \cite{baranykatchalskipach,baranykatchalskipach2}:

\begin{corollary} There exists a universal constant $c(n,k)$ such that, given any system of inequalities $\{x \in \R^n : Ax \leq b\}$, if  every subset of the constraints of
cardinality $c(n,k)$ has more than $k$ integer solutions, then the entire system of inequalities must have more than $k$ integer solutions.
\end{corollary}

One natural question is how to rephrase Theorem \ref{k-doignon} for convex sets rather than systems of linear inequalities

\begin{lemma}\label{finite}
Given an infinite collection of convex sets $( X_i )_{i \in
\Lambda}$, $X_i \subseteq \R^n$ such that there is some index $\tilde{r}$
with $X_{\tilde{r}}$  bounded and $\bigcap_{i \in \Lambda} X_i$ contains exactly
$k$ integer points, there is a finite subcollection of these convex sets
whose intersection contains these $k$ integer points and no others.
\end{lemma}

\begin{proof}
Consider the set  $\bigcap_{i \in \Lambda} X_i.$
Since this is a subset of $X_{\tilde{r}}$, the intersection is properly contained in a hypercube $B$ with integer vertices. Consider the
set
$$S = \{x \in B \setminus \bigcap_{i \in \Lambda} X_i: x \in \Z^n \}.$$
Since $B$ bounds this set, $S$ is finite. Note that $S$ is non-empty otherwise $B=\bigcap_{i \in \Lambda} X_i$, which means that the (finitely many) $X_i$ that contain the facets
of $B$ form the desired finite subcollection. For each $x \in S$,
define $P_x = \{X_i : x \not \in X_i \}$. Using the axiom of choice (if
$P_x$ is uncountable, else enumerate and pick the $X_i$ of least index),
pick an element $X_{i(x)}$ in $P_x$ for each $x \in S$.

Let $I$ be the set of indices $I:=\{ i(x) : x \in S \} \cup \{\tilde{r}\}$.  Note that since $S$ is finite, then $I$ is finite.

Claim: $\bigcap_{{i} \in I} X_{i}$ contains the same integer
points as $\bigcap_{i \in \Lambda} X_i$ and no others.

Obviously   $\bigcap_{i \in \Lambda} X_i \subseteq \bigcap_{i \in I} X_{i}$
since it is the intersection of more sets. For the reverse
containment, assume for contradiction that there is an integer point $y \in
\bigcap_{i \in I} X_{i}$ such that $y \not \in
\bigcap_{i \in \Lambda} X_i$. Obviously $y \in B$, since $\tilde{r} \in I$.
By construction, $I$ contains at least one index $i(x)$ such that $X_i(x)$  excludes $y$, so $y
\not \in \bigcap_{i \in I} X_{i}$. This is a contradiction. Thus

$$\left(\bigcap_{i \in I}  X_{i}\right)\cap \Z^n \subseteq \bigcap_{i \in \Lambda} X_i,$$
so $( X_{i})_{i \in I}  $ is the desired finite subcollection of our original convex sets. \cqfd
\end{proof}

\begin{theorem}\label{convex} Let $n,k$ be positive integers. There is a universal constant $c(n,k)$, depending
only on the dimension $n$ and $k$, such that, for any collection $(X_i)_{i \in \Lambda}$ of closed convex sets in $\R^n$,
where at least one of the sets is bounded and exactly $k$ integer points are in $\bigcap_{i \in \Lambda} X_i$;  then there is a
subcollection of size less than or equal to $c(n, k)$ with exactly the same $k$ integer points in their intersection.
\end{theorem}

\begin{proof}

By Lemma \ref{finite}, it suffices to consider a finite subcollection ${\cal A} \subset \Lambda$ of indices, where at least one
of the $X_i, i \in {\cal A}$ is bounded, say without loss of generality $X_1$. Since $X_1$ is bounded,
there is a hypercube $B$ with integer vertices that bounds it. Let $U$ be the set of
$2n$ hyperplanes that determine $B$. For each of the integer
points $y$ in $B \setminus \bigcap_{i \in {\cal A}} X_i$, there is an
$X_{\alpha}$ such that $y \not \in X_{\alpha}$. It follows that there
is a supporting hyperplane $v_y$ of $X_{\alpha}$ which is violated by
$y$.
Let
$$ S = \bigcup_{y \in (B \cap \Z^n) \setminus \bigcap_{i \in {\cal A}} X_i} \{v_y \}
\cup U.$$
Note that $S$ is finite. Let $P$ be the polytope determined by the
constraints in $S$. Then $P$  contains the $k$ integer points in
$\bigcap_{i \in {\cal A}} X_i$, because the hyperplanes were selected from
 $B$ and the $v_y$s.

For the reverse containment, assume for a contradiction that there is an
integer point $y$ that is in $P$, but not
in $\bigcap_{i \in {\cal A}} X_i$. Obviously, $y \in B$. But by construction,
$S$ contains a constraint which violates $y$, so that $y$ is not in
the polytope determined by $S$.

By Theorem \ref{k-doignon}, $P$ has at most $c(n, k)$
necessary hyperplanes, i.e. there is a subset of the constraints in $S$ of
size no more than $c(n, k)$ that yield a polyhedron $P_{c(n,k)}$  which contains the
$k$ original integer points and no others. By construction, the hyperplanes in
$P_{c(n,k)}$ have a
well-defined identification with the $X_i$'s. Namely, if a hyperplane
$v \in S$ comes from $B$, then it is identified with $X_1$. Otherwise,
$v$ is identified with an arbitrary chosen $X_{\alpha}$ from a finite number of convex set used in the construction of $v$. Call this identification $\phi$. The image of the
hyperplanes in $P_{c(n, k)}$ under $\phi$ is a subcollection $(X_i)_{i \in \Phi}$
which contains the $k$ original integer points and no others, by construction.\cqfd
\end{proof}


 As a final application of Theorem \ref{k-doignon} we are interested in using it within the theory of violator spaces, and then a Clarkson-type randomized algorithm,
to compute the best, 2nd best,$\dots$, $l$-th best solutions to a given integer linear optimization problem.  We call these points the {\em $l$-best solutions of an ILP}.
The literature on this problem is quite extensive and established (see e.g., \cite{hamacherqueyranne,lawler}). As we will see
the resulting algorithm will be linear on the number of constraints, when $l$ and  the dimension are fixed constants.

Let us remember the basics of this theory.
In the years since Clarkson wrote his well-known paper \cite{clarkson}, several researchers observed that his algorithm
works for optimization problems that fit certain abstract structures. This applies for LP-type problems e.g., in \cite{amenta,sharirwelzl}. More
recently G\"artner et al \cite{gaertneretal} proved that in fact  Clarkson's algorithm works in greatest generality for the so called \emph{violator spaces}.
Essentially, a violator space is an optimization problem in which we have a finite set of constraints or elements $H$ and a function that given any subset of constraints $G$, indicates which other constraints in $H \setminus G$ \emph{violate} the feasible solutions to $G$. If one has a violator space structure,
the optimal solution of the problem can be computed via a randomized method whose  running time is {\em linear} in the number of constraints defining the problem,
and {\em subexponential} in the dimension  of the problem. Thus when $l$ and the dimension of the problem are constant,  it gives
a polynomial-time method. We recall here the necessary definitions and properties of violator spaces and see that they fit the problem at hand.

\begin{definition}
A \emph{violator space} is a pair $(H, V)$ where $H$ is a finite set
and $V$ is a mapping $2^H \to 2^H$ such that
the following two conditions hold.
\begin{itemize}
\item
\textbf{Consistency}: $G \cap V(G) = \emptyset$ for all $G \subseteq
H$,
\item
\textbf{Locality:} For all $F \subseteq G \subseteq H$, where $G \cap
V(F) = \emptyset$, we have $V(G) = V(F)$.
\end{itemize}
\label{violatorspace}
\end{definition}

In our case, $H$ is the set of linear inequality constraints of an integer linear program
\begin{math}
IP(H)=\text{min }\{ c^{T} x \mid
 a^{(i)}x \leq b^{(i)}, i\in H,
 x \in \Z^n\}
\end{math}.
For every $G\subseteq H$, we consider the IP defined using a subset of the
constraints \begin{math}
IP(G)=\text{min }\{ c^{T} x \mid
 a^{(i)}x \leq b^{(i)}, i\in G,
 x \in \Z^n\}
\end{math}.

We define the violator set $V(G)$ as the set of inequalities $h\in H$ such that
the $l$ best solutions of $IP(G)$ are not identically equal to the
$l$ best solutions of $IP(G\cup\{h\}).$
Note that we need to have a total ordering on the possible feasible
solutions of $IP(G)$. We therefore provide each integer program with
a universal tie-breaking rule, for instance, using lexicographic ordering.
Assume that $IP(G)$ has at least $l$ different feasible solutions.
Define $\vec{x}_l(G)$ to be the $l$-tuple
consisting of the $l$ best integer points in $IP(G)$ under this
ordering. For our application we say that a constraint $h \in H$ is in $V(G)$ if
$\vec{x}_l(G) \neq \vec{x}_l(G \cup \{h\})$.
If we assume that $IP(G)$ has less than $l$ feasible solutions, we define
$V(G)$ as being the empty set.

We call the pair $(H,V)$ defined above as an $l$-th best IP.
\begin{lemma}
The $l$-th best IP  is a violator space.
\end{lemma}
\begin{proof}
We need to check that the two conditions presented in Definition \ref{violatorspace}
are satisfied. The consistency condition is clearly satisfied.

Assume that $F \subseteq G \subseteq H$ and $G \cap V(F) = \emptyset$.
To show locality we want to show that $V(F) = V(G)$.
We first consider the case where $IP(F)$ has less than $l$ feasible solutions.
Then $V(F)=\emptyset$. Obviously $IP(G)$ has less feasible solutions than $IP(F)$
since $IP(G)$ includes already all constraints of $IP(F)$. Hence $V(G)=\emptyset=V(F).$

In the following, we assume that $IP(F)$ has at least $l$ feasible solutions.
For the  containment $V(F) \subseteq V(G)$ note that $V(F)$ can be alternatively characterized as the set of
constraints that violate or cut at least one  $x_i \in \vec{x}_l(F)$. These constraints cannot be in
$G$, by the assumption that $G \cap V(F) = \emptyset$. It follows that the set of constraints that remove
any of the $x_i \in \vec{x}_l(F)$ cannot be in $G$, so that the same points $x_i \in \vec{x}_l(G)$ for $i=1,...,l$. Thus
any constraint $m \in V(F)$ is in $V(G)$, so that $V(F) \subseteq V(G)$.
For the  containment $V(G) \subseteq V(F)$, observe that when $m \in G$, then $m \not \in V(F)$ (because $G \cap
V(F) = \emptyset$) , thus $\vec{x}_l(F \cup \{m\}) =
\vec{x}_l(F)$. In other words, we can add any constraint in $G
\setminus F$ to $F$ without changing  $\vec{x}_l(F)$. Add all those constraints
$m \in G \setminus F$. We then have $\vec{x}_l(F) = \vec{x}_l(G)$
since each additional constraint left the $l$-tuple unaltered. Because $\vec{x}_l(F) = \vec{x}_l(G)$
any constraint in $V(G)$, also violates at least one of the $x_i \in \vec{x}_l(F)$, so that $V(G) \subseteq V(F)$.
\cqfd
\end{proof}
%
Before we can apply the theory of violator spaces of  \cite{gaertneretal} we need two more ingredients.
First, just like a linear programming optimum is defined by a basis, we need to have a notion of basis for our optimal solutions.
\begin{definition}
Given a violator space $(H, V)$, we say that $B \subseteq H$ is a
\emph{basis}  if for all proper subsets $F \subset B$ we have $B
\cap V(F) \neq \emptyset$. For $G \subseteq H$, a basis of $G$ is a
minimal subset $B$ of $G$ that is a basis  and such that  $V(B) = V(G)$.

The \emph{combinatorial dimension} of a violator space $(H,V)$ is the maximal
cardinality of a basis.
\end{definition}
Intuitively a basis is a  minimal subset of constraints with the same optimal value as the whole set. Now we see that the size of
a basis is bounded by the constant of Theorem \ref{k-doignon}.
\begin{lemma}
The combinatorial dimension of $l$-th best IP with $n$ variables is bounded by the constant $c(n,l).$
\end{lemma}
\begin{proof}
Consider a set of constraints $H$ and the corresponding $IP(H)$.
First assume that $IP(H)$ has $\bar l<l$ feasible  solutions.
Then there exists a subset of the constraints $G\subseteq H$ with cardinality at most
$c(n,\bar l)$ that has the same number of solutions and is therefore a basis.
Observe that $c(n,\bar l)\leq c(n,l).$

Assume now that $IP(H)$ has at least $l$ feasible solutions.
Considering the set of constraints $\{ c^T x \leq c^T x_l \} \cup H$,
where $x_l$ is the $l$-th best value in $\vec{x}_l(H)$, and applying
Theorem \ref{k-doignon}, we conclude that there exists a subset of the constraints
of cardinality at most $c(n,l)$ that define the same $l$ feasible solutions.
Getting rid of $\{ c^T x \leq c^T x_l \}$, we conclude that the size of the largest basis
of $H$ is $c(n,l)$.\cqfd
\end{proof}

The second ingredient is that we need to have a way to answer the following query in polynomial time in fixed dimension
for subsets of size smaller than the combinatorial dimension.

\vskip .2cm

\noindent {\bf Primitive query:} Given $G \subset H$ and $h \in H \setminus G$, decide whether $h \in V(G)$.

\vskip .2 cm

The reason we need to answer this query via a black-box method is because  then, using Theorem 27 in Section 4 of  \cite{gaertneretal}, we obtain the following result

\begin{lemma}
A basis of H in a violator space (H,V) (and thus an optimal solution to the problem) can be found by
calling the algorithm  that solves the primitive query an expected
$$O(c(n,l)|H|+ c(n,l)^{O(c(n,l))})$$ number of times.
\end{lemma}

The primitive query originally provided by Clarkson in the case of regular integer programming was Lenstra's IP algorithm in fixed dimension.
For our problem, when we now look for the $l$-th best solution, the primitive query can be answered
by calling $l$ times an algorithm for IP in fixed dimension.
Given an integer program in fixed dimension $n$, a fixed number of constraints $q$ and a  maximum bit size of the data of $s$,
Eisenbrand provided in  \cite{eisenbrandfixeddim} an algorithm that finds an optimal
solution in $\mathcal O(s)$ operations.
Our primitive query calls a maximum of $l$ times Eisenbrand's algorithm in fixed dimension $n$ and constant number of constraints $c(n,l)$ to answer the question.
We make this explicit in the following lemma:

\begin{lemma}
Given fixed positive integers $q,n$ and $l$ and an integer program $\min\{c^Tx: Ax\leq b, x\in \Z^n, 0\leq x \leq u\}$ with $q$
constraints and a varying maximum encoding length of $s$ bits for the data, it is possible to determine in $\mathcal O(ls)$ operations
the $l$ best solutions to the integer program where ties are broken using lexicographic order.
\end{lemma}
\begin{proof}
We use the global bound $u$ on the variables, define $\bar c:= u^nc+\sum_{i=1}^n u^{n-i} e_i$  and  set up the auxiliary integer program
$\min \{\bar c^Tx: Ax\leq b, x\in \Z^n, 0\leq x\leq u\}.$
Observe that by solving the latter, we have the guarantee to find the best integer point with respect to the objective function
breaking ties lexicographically.
Denoting $\bar x^{(1)}$ the optimal solution to the above integer program, we can now find the second best integer
point by solving
$\min\{\bar c^Tx : Ax\leq b, x\in \Z^n, \bar c^Tx\geq \bar c^T\bar x^{(1)}+1,\ 0 \leq x\leq u\}.$
If we denote by $\bar x^{(i)}$ the $i^{th}$ best point, we obtain the $(i+1)^{th}$ best point by solving a similar integer
program with the additional constraint $\bar c^Tx\geq \bar c^Tx^{(i)}+1.$

This shows that we are able to compute the $l$ best solutions to an integer program by successively solving $l$ integer programs
in dimension $n$ with $q+1$ constraints and with bit size $ns$ which is of order $s$ when $n$ is fixed. The result follows by applying Eisenbrand's algorithm
to all integer programs.\cqfd
\end{proof}

We finally arrive at the key complexity consequence of this section:

\begin{corollary}
Given fixed positive integer constants $n$ and $l$, an integer $m \times n$ matrix $A$, and the integer linear program
\begin{align*}
\text{min } & c^T x \\
\text{subject to }& Ax \leq b \\
x \in \Z^n
\end{align*}
of  maximum bit-size $s$ for the coefficients in $A$, then the $l$-th best solution can be computed in a expected

$$O( (ls)(c(n,l)m+ c(n,l)^{O(c(n,l))})$$ number of operations.
\end{corollary}

We conclude by noting that using the theory of rational generating functions (see \cite{AGTO} for an introduction) one can also prove a similar result with some worse complexity.

\section*{Acknowledgements} The research of the third author was partially supported by a UCMEXUS project grant and by the Institute for Mathematics and its Applications with funds provided by the National Science Foundation. The research was carried out during the fourth author's sabbatical visit at UC Davis supported by the Belgian Science Foundation (FRS-FNRS) and UC Davis. The authors are truly grateful to the anonymous referees for their excellent detailed comments and corrections that greatly improved the quality of this paper. We are also grateful to Robert Hildebrand, Luis Montejano, Timm Oertel, Deborah Oliveros, J\'anos Pach, Edgardo Rold\'an Pensado, and Stefan Weltge for their very useful comments.


\bibliographystyle{plain}

\end{document}